\documentclass[11pt]{amsart}

\usepackage{amscd,amssymb,amsopn,amsmath,amsthm,mathrsfs,graphics,amsfonts,enumerate,verbatim,calc
}
\usepackage{bbm}
\usepackage[all,cmtip]{xy}
\usepackage[all]{xy}
\usepackage{tikz}
\usepackage{lscape}
\usepackage{enumitem}
\usetikzlibrary{matrix,arrows,decorations.pathmorphing,cd}
\usepackage{hyperref}
\hypersetup{colorlinks=true,linkcolor={blue}}
\usepackage{cleveref}
\usepackage{url}

\usepackage{scalerel}
\usepackage{stackengine,wasysym}
\usetikzlibrary {positioning}
\usetikzlibrary{patterns}
\usetikzlibrary{calc}
\definecolor {processblue}{cmyk}{0.96,0,0,0}

\usepackage{subfigure}

\usepackage{comment} 

\usepackage{ dsfont }

\usepackage{color}

\usepackage[OT2,OT1]{fontenc}
\newcommand\cyr{%
\renewcommand\rmdefault{wncyr}%
\renewcommand\sfdefault{wncyss}%
\renewcommand\encodingdefault{OT2}%
\normalfont
\selectfont}
\DeclareTextFontCommand{\textcyr}{\cyr}

\usepackage{amssymb,amsmath}

\DeclareFontFamily{OT1}{rsfs}{}
\DeclareFontShape{OT1}{rsfs}{n}{it}{<-> rsfs10}{}
\DeclareMathAlphabet{\mathscr}{OT1}{rsfs}{n}{it}

\numberwithin{equation}{section}
\newtheorem{theorem}{Theorem}[section]
\newtheorem{lem}[theorem]{Lemma}
\newtheorem{prop}[theorem]{Proposition}

\theoremstyle{definition}
\newtheorem{defn}[theorem]{Definition}
\theoremstyle{remark}
\newtheorem{remark}[theorem]{Remark}

\newtheorem{example}[theorem]{Example}

\newcommand{\Ass}{\operatorname{Ass}}

\newcommand{\im}{\operatorname{im}}
\renewcommand{\ker}{\operatorname{ker}}

\newcommand{\Ext}{\operatorname{Ext}}

\newcommand{\Hom}{\operatorname{Hom}}

\newcommand{\End}{\operatorname{End}}

\newcommand{\soc}{\operatorname{soc}}
\newcommand{\coker}{\operatorname{coker}}

\newcommand{\rank}{\ensuremath{\operatorname{rank}}}

\newcommand{\p}{\mathfrak{p}}

\newcommand{\m}{\mathfrak{m}}

\newcommand{\w}{\omega}
\renewcommand{\bar}{\overline}

\begin{document}
\title[Endomorphism Algebras over Commutative Rings and Torsion in Self Tensor Products]{Endomorphism Algebras Over Commutative Rings and Torsion in Self Tensor Products}

\author[Lyle]{Justin Lyle}
\email[Justin Lyle]{jll0107@auburn.edu}
\urladdr{https://jlyle42.github.io/justinlyle/}
\address{Department of Mathematics and Statistics \\ 221 Parker Hall\\
	Auburn University\\
	Auburn, AL 36849}



\begin{abstract}

Let $R$ be a commutative Noetherian local ring. We study tensor products involving a finitely generated $R$-module $M$ through the natural action of its endomorphism ring. In particular, we study torsion properties of self tensor products in the case where $\End_R(M)$ has an $R^*$-algebra structure, and prove that if $M$ is indecomposable, then $M \otimes_{\End_R(M)} M$ must always have torsion in this case under mild hypotheses. 
    
\end{abstract}

\maketitle

\section{Introduction}

Tensor products are among the most ubiquitous objects in commutative algebra, yet they possess many aspects that remain mysterious. For instance, there is a general philosophy that tensor products of finitely generated modules over a local domain should be expected to have torsion, or should only have good depth properties under stringent hypotheses. Nonetheless, examples exist that show this philosophy does not hold in general, and so numerous open conjectures and problems have emerged that give a precise foothold to this intuition, with much work being put toward them; see \cite{Au61,CG19,CI15,GT15,HI19,HW94,HW97} for a few examples. 

In particular, while it is well-known that $I \otimes_R I$ must always have torsion when $I$ is an ideal of positive grade in a local ring $R$, a construction due to Huneke-Wiegand shows that any non-Gorenstein complete local domain $R$ of dimension $1$ admits a finitely generated torsion-free $R$-module $M$ that is not free but for which $M \otimes_R M$ remains torsion-free \cite[Proposition 4.7]{HW94}. Motivated by these examples, and viewing the torsion properties of $I \otimes_R I$ as being owed in part to the commutativity of $\End_R(I)$, we study self tensor products involving a finitely generated $R$-module $M$ by exploiting the natural left module structure of $M$ over its endomorphism ring. When such tensor products can be calculated over $\End_R(M)$ instead of $R$, they tend to be smaller and more manageable. Leveraging this insight, we consider torsion properties of the $R$-module $M \otimes_{\End_R(M)} M$ in the situation when $\End_R(M)$ possesses an $R^*$-algebra structure. Our main theorem shows that the behavior of the ideal setting extends to indecomposable modules of higher rank in this situation under mild hypotheses, and reveals the torsion-free property of $M \otimes_{\End_R(M)} M$ as an often more relevant obstruction than that of $M \otimes_R M$ (see Theorem \ref{mainthm2}):

\begin{theorem}\label{introthm2}
Suppose $R$ is a Henselian domain and that $M$ is an indecomposable torsion-free $R$-module. Suppose $\End_R(M)$ has an $R^*$-algebra structure, and equip $M$ with the induced right $\End_R(M)$-module action through this structure. If $M \otimes_{\End_R(M)} M$ is a torsion-free $R$-module, then $M$ is a cyclic $E$-module. 
\end{theorem}
We also provide an extended example (Example \ref{selfdual}) that serves as a counterpoint to the aforementioned construction of Huneke-Wiegand, and shows that even though $M$ and $M \otimes_R M$ are torsion-free in their examples, $M \otimes_{\End_R(M)} M$ may still have torsion.

\section{Preliminaries}

Throughout, we let $(R,\m,k)$ be a commutative Noetherian local ring. All modules are assumed to be finitely generated. We note that an $R$-module $M$ carries a natural left module action of $\End_R(M)$, given by $f \cdot x=f(x)$, and $\Hom_R(M,N)$ carries the structure of an $\End_R(N)-\End_R(M)$ bimodule by composition in the natural way. We recall the $R$-module $M$ is said to have \emph{rank} $r$ if $M_{\p} \cong R_{\p}^{\oplus r}$ for all $\p \in \Ass_R(R)$, equivalently if $M \otimes_R Q$ is a free $Q$-module of rank $r$ where $Q$ denotes the total quotient ring of $R$. We let $\mu_R(M)$ denote the minimal number of generators of $M$ over $R$ and we write $(-)^{\dagger}:=\Hom_R(-,R)$.

We recall the following observation due to Huneke-Wiegand:
\begin{prop}[{\cite[Proposition 4.7]{HW94}}]\label{hwex}
Suppose $R$ is a local domain of dimension $1$ with canonical module $\w$ that is not Gorenstein. Then there exists a nonfree torsion-free $R$-module $M$ so that $M \otimes_R M$ is torsion-free.
\end{prop}
A counterpoint to this observation is the case where $I$ is an ideal of positive grade; in this case it is well-known that $I \otimes_R I$ can only be torsion-free if $I$ is principal. It's also well-known (see e.g. \cite[Exercise 4.3.1]{LW12}) that $\End_R(I)$ is commutative when $I$ has positive grade, and from the injection of rings $R \to \End_R(I)$, we have a surjection $I \otimes_R I \to I \otimes_{\End_R(I)} I$. By rank considerations, the kernel of this surjection must be torsion. So if $I \otimes_R I$ is torsion-free, then this surjection is an isomorphism and $I \otimes_{\End_R(I)} I$ is torsion-free as well. 

In view of Proposition \ref{hwex}, we interpret that the torsion-freeness of $I \otimes_{\End_R(I)} I$ is an often more relevant obstruction to consider, and we seek to exploit this observation for modules of higher rank. To even be able to form the $R$-module $M \otimes_{\End_R(M)} M$, we need to work in a situation in which $M$ possesses a right module structure over $\End_R(M)$. This is guaranteed by the commutativity of $\End_R(I)$ when $I$ is an ideal of positive grade, but if $M$ has rank $r>1$, then $\End_R(M)$ cannot be commutative, as $\End_R(M)_{\p} \cong \End_{R_{\p}}(M_{\p}) \cong M_r(R_{\p})$, which is noncommutative for all $\p \in \Ass_R(R)$. We thus consider instead the following:

\begin{defn}

A (possibly noncommutative) $R$-algebra $E$ is said to be an $R^*$-algebra if there is an $R$-linear ring involution $(-)^*:E \to E^{op}$. In other words, $*$ is a map of abelian groups such that
\begin{enumerate}
\item[$(1)$] $(ab)^*=b^*a^*$ for all $a,b \in E$.

\item[$(2)$] $1_E^*=1_E$.

\item[$(3)$] $(ra)^*=ra^*$ for all $r \in R$ and $a \in E$.

\item[$(4)$] $a^{**}=a$ for all $a \in E$.

\end{enumerate}
    
\end{defn}

$R^*$-algebras, while usually noncommutative, share many of the module theoretic properties of a commutative algebra. In particular, if $E$ is an $R^*$-algebra and $M$ is a left $E$-module, then $M$ inherits a right $E$-module structure given by $x \cdot f=f^*x$ for $x \in M$ and $f \in E$. One should take caution that these actions are not compatible in the sense that they do not induce an $(E,E)$ bimodule structure on $M$ in general. However, if we take the natural $(E,R)$ bimodule structure induced by the left action of $E$ on $M$, then there is an induced $(R,R)$ bimodule structure inherited from this right action of $E$ on $M$ that agrees with the natural one owed to the commutativity of $R$. For clarity, for a given left $E$-module $M$, we will write $M_*$ to indicate $M$ viewed as a right module under this induced structure. 

\begin{example}\label{exs}

The following are some motivating examples of $R^*$-algebras:

\begin{enumerate}

\item[$(1)$] Any commutative $R$-algebra $S$ where $(-)^*$ is the identity. Of particular note $\End_R(I)$ when $I$ is an ideal of positive grade.

\item[$(2)$] $M_n(R)$ for any $n$ where $(-)^*$ is the matrix transpose.

\item[$(3)$] If $R=\mathbb{R}$, then $\mathbb{C}$ under complex conjugation.

\item[$(4)$] If $E$ and $S$ are $R^*$-algebras under $(-)^{*_E}$ and $(-)^{*_S}$ respectively, then $E \otimes_R S$ is an $R^*$-algebra under $(-)^{*}$ given by $(x \otimes s)^{*}=x^{*_E} \otimes s^{*_S}$. In particular, if $S$ is commutative, e.g, a localization or completion of $R$, then $E \otimes_R S$ then $(-)^{*}$ may be given by $(x \otimes s)^{*}=x^{*_E} \otimes s$.

\end{enumerate}

\end{example}

Another significant source of $R^*$-algebras is given by the following:

\begin{defn}

We say an $R$-module $M$ is \emph{strongly self dual} with respect to an $R$-module $N$ if there is an $R$-module isomorphism $\alpha:M \to \Hom_R(M,N)$ such that $\alpha(x)(y)=\alpha(y)(x)$ for all $x,y \in M$.

\end{defn}

\begin{prop}\label{stronglyself}

If $M$ is strongly self dual with respect to $N$ then $\End_R(M)$ has an $R^*$-algebra structure so that $M_* \cong \Hom_R(M,N)$ as right $\End_R(M)$-modules.

\end{prop}

\begin{proof}

Suppose $M$ is strongly self dual with respect to $N$ so that there is an $R$-module isomorphism $\alpha:M \to \Hom_R(M,N)$ satisfying $\alpha(x)(y)=\alpha(y)(x)$ for all $x,y \in M$. Define $(-)^{*}:\End_R(M) \to \End_R(M)^{op}$ by $f^{*}=\alpha^{-1} \circ \Hom_R(f,N) \circ \alpha$. It is clear that $(-)^{*}$ is an $R$-linear ring homomorphism. We claim it is an involution. We have 
\[f^{**}=\alpha^{-1} \circ \Hom_R(f^*,N) \circ \alpha=\alpha^{-1} \circ \Hom_R(\alpha^{-1} \circ \Hom_R(f,N) \circ \alpha^{-1},N) \circ \alpha\]
\[=\alpha^{-1} \circ \Hom_R(\alpha,N) \circ \Hom_R(\Hom_R(f,N),N) \circ \Hom_R(\alpha^{-1},N) \circ \alpha.\]
We observe for any $x \in M$ that $(\Hom_R(\alpha^{-1},N) \circ \alpha)(x)=\alpha(x) \circ \alpha^{-1}$. But if $g=\alpha(y) \in \Hom_R(M,N)$, then $\alpha(x)(\alpha^{-1}(g))=\alpha(x)(y)=\alpha(y)(x)=g(x)$. In other words, $\Hom_R(\alpha^{-1},N) \circ \alpha: M \to \Hom_R(\Hom_R(M,N),N)$ is the natural biduality map, and it is well known (and easy to see) that the diagram
\[\begin{tikzcd}[column sep=4em,row sep=large]
	M && M \\
	{\Hom_R(\Hom_R(M,N),N)} && {\Hom_R(\Hom_R(M,N),N)}
	\arrow[from=1-3, to=2-3]
	\arrow[from=1-1, to=2-1]
	\arrow["{\Hom_R(\Hom_R(f,N),N)}", from=2-1, to=2-3]
	\arrow["f", from=1-1, to=1-3]
\end{tikzcd}\]
whose vertical arrows are the biduality map, commutes for any $f \in \End_R(M)$. So $f^{**}=f$ for all $f \in \End_R(M)$, and therefore $(-)^{*}$ induces an $R^*$-algebra structure on $\End_R(M)$. We claim $\alpha:M_* \to \Hom_R(M,N)$ is right $\End_R(M)$-linear. Indeed, $\alpha(x \cdot f)=\alpha(f^*(x))=\alpha(\alpha^{-1}(\Hom_R(f,N)(\alpha(x))))=\Hom_R(f,N)(\alpha(x))=\alpha(x) \circ f$, so $\alpha:M_* \to \Hom_R(M,N)$ gives an isomorphism of right $\End_R(M)$-modules, as desired. 
\end{proof}

\begin{remark}
The notion of strongly self dual modules has appeared in various guises in the literature. For instance, an Artinian local ring $R$ is said to be \emph{Teter} if can be expressed as $S/\soc(S)$ for an Artinian Gorenstein local ring $S$. The original work of Teter on these rings characterizes Teter rings as those Artinian local rings for which $\m$ is strongly self dual with respect to $E_R(k)$; see \cite[Lemma 1.1 and Theorem 2.3]{Te74}. In particular, a consequence of Proposition \ref{stronglyself} is that $\End_R(\m)$ always has an $R^*$-algebra structure whenever $R$ is Teter. Another concrete example is given in Example \ref{selfdual} and the pushforward construction there can be used to produce numerous other examples as well.
\end{remark}

\section{Main Results}\label{rstar}

In this section we provide the proof of Theorem \ref{introthm2}. We will need a few preparatory results, beginning with the following that is inspired from \cite{Is80}:

\begin{prop}\label{*algstruc}
Suppose $R$ is a Noetherian semi-local ring. A map $(-)^*: M_n(R) \to M_n(R)^{op}$ induces an $R^*$-algebra structure on $M_n(R)$ if and only if there is an invertible matrix $C \in M_n(R)$ and an $a \in R$ so that 
\begin{enumerate}
\item[$(1)$] $C^T=aC$.
\item[$(2)$] $a^2=1$.
\item[$(3)$] $(-)^*$ is given by $A^*=C^{-1}A^TC$.
    
\end{enumerate}

\end{prop}

\begin{proof}

It is clear that any such $(-)^*$ with $C$ and $a$ satifying conditions $(1)-(3)$ gives an $R^*$-algebra structure. Conversely, suppose $(-)^*$ induces an $R^*$-algebra structure, and let $\alpha:M_n(R) \to M_n(R)$ be given by $\alpha(A)=(A^T)^*$. Then $\alpha$ is an $R$-algebra automorphism of $M_n(R)$. Note that if $M^{\oplus n}$ is a free $R$-module of rank $n$ for some $R$-module $M$, then as $R$ is semi-local, it must be that $M \cong R$. Thus by \cite[Theorem 16]{Is80}, it follows that $\alpha$ is an inner automorphism. Then there is an invertible matrix $C$ so that $\alpha(A)=C^{-1}AC$ for all $A \in M_n(R)$. But this means $\alpha(A^T)=A^*=C^{-1}A^TC$ for all $A \in M_n(R)$. But then
\[A=A^{**}=C^{-1}(A^*)^TC=C^{-1}C^TA(C^T)^{-1}C\]
which implies $AC^{-1}C^T=C^{-1}C^TA$ for all $A \in M_n(R)$. Thus $C^{-1}C^T$ is is the center of $M_n(R)$, so $C^{-1}C^T=aI_n$ for some $a \in R$. Then $C^T=aC$. But then $C=(aC)^T=aC^T=a^2C$. As $C$ is invertible, we have $a^2=1$. 
\end{proof}

\begin{lem}\label{localtorsion}
If $(-)^*:M_n(R) \to M_n(R)^{op}$ induces an $R^*$-algebra structure on $M_n(R)$, so that by Proposition \ref{*algstruc}, $(-)^*$ is given by $A^*=C^{-1}A^TC$ with $C^T=aC$ and $a^2=1$, then the map $\theta:R^{n \times 1}_* \otimes_{M_n(R)} R^{n \times 1} \to R$ given by $\theta(x \otimes y)=x \cdot Cy$ is an isomorphism of $R$-modules (where $\cdot$ denotes the usual dot product). In particular, for any $x,y \in R^{n \times 1}$, we have $x \otimes y-ay \otimes x=0$ as an element of $R^{n \times 1}_* \otimes_{M_n(R)} R^{n \times 1}$. 

\end{lem}

\begin{proof}

Define $P:R^{n \times 1}_* \times R^{n \times 1} \to R$ by $(x,y) \mapsto x \cdot Cy$. We check that $P$ is $M_n(R)$ bilinear, noting it is obviously biadditive. If $A \in M_n(R)$, we observe
\[P(xA,y)=P(C^{-1}A^TCx,y)=C^{-1}A^TCx \cdot Cy=x \cdot C^TA(C^T)^{-1}Cy\]
\[=x \cdot aCAaC^{-1}Cy=x \cdot CAy=P(x,Ay).\]
Thus we have an induced $R$-linear map $\theta:R^{n \times 1}_* \otimes_{M_n(R)} R^{n \times 1} \to R$ given by $\theta(x \otimes y)=x \cdot Cy$. Note that $\theta$ is surjective as $e \otimes C^{-1}e \mapsto 1$, where $e$ is the column vector with a $1$ as its first entry and zeroes elsewhere. But as is well-known (see for instance \cite[Theorem 17.20]{La99}), the functor $- \otimes_{M_n(R)} R^{n \times 1}$ is part of a Morita equivalence between $M_n(R)$ and $R$, whose inverse functor is given by $- \otimes_R R^{1 \times n}$. As $R^{n \times 1}_*$ is $n$-generated as an $R$-module, it must thus be that $R^{n \times 1}_* \otimes_{M_n(R)} R^{n \times 1}$ is a cyclic $R$-module, and it follows the surjection $\theta$ must be an isomorphism. We then observe 
\[\theta(x \otimes y-ay \otimes x)=x \cdot Cy-ay \cdot Cx=x \cdot Cy-Cx \cdot ay=x \cdot Cy-x \cdot aC^Ty=x \cdot Cy-x \cdot Cy=0.\]
It follows that $x \otimes y-ay \otimes x=0$ as $\theta$ is injective.
\end{proof}

For the remainder of this section, we will suppose $M$ is an $R$-module of constant rank $n$ such that $E:=\End_R(M)$ is an $R^*$-algebra. Let $Q$ be the total quotient ring of $R$. Localizing at the set of nonzerodivisors $S$ of $R$, $E_S \cong M_n(Q)$ inherits a $Q^*$-algebra structure from the $R^*$-algebra structure of $E$. When it is clear from context, we will abuse notation and write $(-)^*$ for either the $R^*$-operation of $E$ or the induced $Q^*$-operation of $E_S$. By Proposition \ref{*algstruc}, $(-)^*:M_n(R) \to M_n(R)^{op}$ is given by $A^*=C^{-1}A^TC$ for some invertible $C \in M_n(R)$ with $C^T=aC$ where $a$ is a unit of $Q$ with $a^2=1$. Write $a=\dfrac{r}{s}$ where $r,s$ are nonzerodivisors of $R$.

\begin{lem}\label{torsion}

Consider the submodule $T=\langle sx \otimes y-ry \otimes x \mid x,y \in M \rangle$ of $M_* \otimes_E M$. Then $T$ is a torsion $R$-module.

\end{lem}

\begin{proof}

Localizing at the set $S$ of nonzerodivisors of $R$, Lemma \ref{localtorsion} implies that $\dfrac{x}{1} \otimes \dfrac{y}{1}-a\dfrac{y}{1} \otimes \dfrac{x}{1}=0$ in $(M_S)_* \otimes_{E_S} M_S$ for any $x,y \in M$. Then so is $\dfrac{sx}{1} \otimes \dfrac{y}{1}-\dfrac{ry}{1} \otimes \dfrac{x}{1}$. Thus $T \otimes_R Q=0$ so $T$ is torsion.
\end{proof}

\begin{lem}\label{genset}

Suppose $x,y$ are part of an minimal generating set for $M$ (resp $M_*$) as an $E$-module. Then $x,y$ are part of a minimal generating set for $M_*$ (resp $M$) as an $E$-module.

\end{lem}

\begin{proof}

It suffices to show $x_1,\dots,x_n$ is a generating set for $M$ if and only if it is a generating set for $M_*$. But this claim follows as a given $g \in M$ can be expressed as $g=\sum_{i=1}^n f_ix_i$ if and only if it can be expressed as $g=\sum_{i=1}^n x_if^*_i$ in $M_*$.
\end{proof}

In light of Lemma \ref{genset}, we may uambiguously speak of ``a minimal $E$-generating set for $M$" without consideration of whether we mean the left or right $E$-module structure of $M$.

\begin{lem}\label{nonzero}

Suppose $E$ is local, e.g. $R$ is Henselian and $M$ is indecomposable. If $x,y$ are part of an $E$-minimal generating set for $M$, then $rx \otimes y-sy \otimes x \ne 0$ as an element of $M_* \otimes_E M$ for any $s \in R$ and any unit $r \in R$.

\end{lem}

\begin{proof}

Let $J(E)$ be the Jacobson radical of $E$, which is maximal as $E$ is local. Set $D=E/J(E)$ which is thus a division ring. Set $V=M/J(E)M$ and $V_*=M_*/M_*J(E)$, and let $p:M \to V$ and $p_*:M_* \to V_*$ be the natural projection maps. Let $x_1,x_2,\dots,x_n$ be a minimal $E$-generating set for $M$ with $x_1=x$ and $x_2=y$. By Nakayama-Azumaya-Krull, the images $\bar{x}_1,\bar{x}_2,\dots,\bar{x}_n$ form a left basis for $V$ and a right basis for $V_*$ over $D$. We thus have the natural Kronecker delta maps $\delta_i:V \to D$ and $\delta^*_i:V_* \to D$ for each $i$, as well as the product map $\beta:D \otimes_D D \to D$ given by $a \otimes b \mapsto ab$. If $\theta_{ij}=\beta \circ ((\delta_i^* \circ p_*) \otimes_D (\delta_j \circ p))$, then $\theta_{12}(r(x \otimes y)-s(y \otimes x))=\bar{r}$. But $\bar{r} \ne 0$, since $r$ is a unit of $R$, and multiplication by $r$ thus corresponds to a unit of $E$. Thus $rx \otimes y-sy \otimes x \ne 0$. 
\end{proof}

We are now ready to present the main theorem of this section.

\begin{theorem}\label{mainthm2}
Suppose $E$ is local, e.g. $R$ is Henselian and $M$ is an indecomposable, that $E$ possesses an $R^*$-algebra structure, and that $M$ is a torsion-free $R$-module with rank. Further suppose one of the following holds:
\begin{enumerate}

\item[$(1)$] $R$ is a domain.

\item[$(2)$] $R$ is reduced of characteristic $2$.

\end{enumerate}

If the $R$-module $M_* \otimes_E M$ is torsion-free, then $M$ is a cyclic $E$-module.

\end{theorem}

\begin{proof}

If $R$ is a domain, then $Q$ is a field, and as $a^2=1$, we must have $a=1$ or $a=-1$. If instead, $R$ is reduced of characteristic $2$, then so is $Q$. As $(a-1)^2=a^2-1=0$ in $Q$, and as $Q$ is reduced, it must be that $a=1$. Thus in either of the two cases, we have $a=1$ or $a=-1$.

Now, consider the $R$-module $T$ of $M_* \otimes_E M$ given by $T=\langle x \otimes y-y \otimes x \mid x,y \in M \rangle$ if $a=1$ or $T=\langle x \otimes y+y \otimes x \mid x,y \in M \rangle$ if $a=-1$. Then Lemma \ref{torsion} gives that that $T$ is torsion, and as $M_* \otimes_E M$ is supposed to be torsion-free, we must have that $T=0$. But by Lemma \ref{nonzero}, this forces $M$ to be a cyclic $E$-module.
\end{proof}

Finally, we return to the construction of Huneke-Wiegand used to prove Proposition \ref{hwex} and use it to provide the following example. Our example shows that while $M$ and $M \otimes_R M$ are torsion-free in the exmaples of Huneke-Wiegand, $M_* \otimes_{\End_R(M)} M$ may still have torsion, even if $M$ and $M_*$ are cyclic $E$-modules.

\begin{example}\label{selfdual}
Consider the numerical semigroup ring $k[\![t^3,t^4,t^5]\!]$. It can be seen from \cite{He70} that this ring has the presentation $R:=k[\![x,y,z]\!]/(y^2-xz,x^2y-z^2,x^3-yz)$, and the canonical ideal $I:=(x,y)$ of $R$ has presentation matrix $\begin{pmatrix} y & z & x^2 \\ -x & -y & -z \end{pmatrix}$. Let $(-)^{\vee}:=\Hom_R(-,I)$. Since $I$ is a two-generated ideal containing the nonzerodivisor $x$, there is an isomorphism $(x:I) \to \Omega^1_R(I)$ given by $a \mapsto \begin{pmatrix} \dfrac{ay}{x} \\ -a \end{pmatrix}$ (see e.g. \cite[Proof of Lemma 3.3]{HH05}). Compose this map with the natural inclusion $\Omega^1_R(I) \to R^{\oplus 2}$ to get an injection $j:(x:I) \to R^{\oplus 2}$. We verify directly that $yI \subseteq (x)$ and $zI \subseteq (x)$, so $(x:I)=\m:=(x,y,z)$. We also note that $(x:\m)=\m$. We thus have a short exact sequence
\[0 \rightarrow \m \xrightarrow{j} R^{\oplus 2} \xrightarrow{p} I \rightarrow 0\]
where $p$ maps the standard basis elements $e_1 \mapsto x$ and $e_2 \mapsto y$. There is also a commutative diagram:
\[\begin{tikzcd}
	0 & I^{\dagger} & (R^{\oplus 2})^{\dagger} & {\m^{\dagger}} & {\Ext^1_R(I,R)} & 0 \\
	0 & \m & {R^{\oplus 2}} & {(x:\m)} & {(x:\m)/I} & 0
	\arrow[from=1-1, to=1-2]
	\arrow["p^{\dagger}", from=1-2, to=1-3]
	\arrow["j^{\dagger}", from=1-3, to=1-4]
	\arrow["\theta", from=1-4, to=1-5]
	\arrow[from=1-5, to=1-6]
	\arrow["j", from=2-2, to=2-3]
	\arrow["{\epsilon \circ p}", from=2-3, to=2-4]
	\arrow[from=2-4, to=2-5]
	\arrow[from=2-5, to=2-6]
	\arrow[from=2-1, to=2-2]
	\arrow["\alpha"', from=1-2, to=2-2]
	\arrow["\beta"', from=1-3, to=2-3]
	\arrow["\gamma"', from=1-4, to=2-4]
	\arrow[dashed, from=1-5, to=2-5]
\end{tikzcd}\]
where $\alpha:I^{\dagger} \to (x:I)=\m$ is the isomorphism given by $f \mapsto f(x)$, $\gamma:m^{\dagger} \to (x:\m)=\m$ is the isomorphism given by $f \mapsto f(-x)$, where $\beta:(R^{\oplus 2})^{\dagger} \to R^{\oplus 2}$ is the isomorphism given by $f \mapsto \begin{pmatrix} f(e_2) \\ -f(e_1) \end{pmatrix}$, and where $\epsilon:I \to (x:\m)=\m$ is the natural inclusion. As $(x:\m)=\m$, we see $(x:\m)/I \cong z/z \cap I$. As $z\m \subseteq (x) \subseteq I$, it follows that $(x:\m)/I$ is generated by the image of $z$ and that $(x:\m)/I \cong k$. In particular, the map $i:\m \to R$ given as multiplication by $\dfrac{z}{x}$ maps under $\theta$ to a nonzero element of $\Ext^1_R(I,R)$. But by definition, $\theta(i)$ is the extension of $I$ by $R$ given as the pushforward along the maps $j$ and $i$. So there is a pushforward diagram
\[\begin{tikzcd}
	0 & \m & {R^{\oplus 2}} & I & 0 \\
	0 & R & M & I & 0
	\arrow[from=1-1, to=1-2]
	\arrow["j", from=1-2, to=1-3]
	\arrow["p", from=1-3, to=1-4]
	\arrow[from=1-4, to=1-5]
	\arrow[from=2-1, to=2-2]
	\arrow[from=2-4, to=2-5]
	\arrow["i"', from=1-2, to=2-2]
	\arrow[from=1-3, to=2-3]
	\arrow[equal, from=1-4, to=2-4]
	\arrow["s", from=2-2, to=2-3]
	\arrow["t", from=2-3, to=2-4]
\end{tikzcd}\]
with exact rows, with $M=\coker(A)$ for $A=\begin{pmatrix} y & z & x^2 \\ -x & -y & -z \\ z & x^2 & xy \end{pmatrix}$, where $s:R \to M=R^{\oplus 3}/\im A$ is given by $s(1)=\bar{e}_3$, and where $t:M \to I$ is given by $t(\bar{e}_1)=x$, $t(\bar{e}_2)=y$ and $t(\bar{e}_3)=0$. In particular, the bottom row does not split, and we note that the depth lemma applied to the bottom row forces $M$ to be torsion-free. We also note that $\mu_R(M)=3$, and by additivity of rank, that $\rank(M)=2$. Since $I$ is a canonical ideal, applying $\Hom_R(-,I)$ to the bottom row, we obtain an extension of the form $0 \to R \to M^{\vee} \to I \to 0$. In particular, $\mu_R(M^{\vee}) \le 3$.

We will produce a concrete generating set for $M^{\vee}$ and a concrete isomorphism $M \to M^{\vee}$. Define maps $g_1,g_2,g_3:R^{\oplus 3} \to I$ by $g_1(e_1)=-y$, $g_1(e_2)=0$, $g_1(e_3)=x$, $g_2(e_1)=0$, $g_2(e_2)=x^2$, $g_2(e_3)=y$, $g_3(e_1)=x$, $g_3(e_2)=y$, and $g_3(e_3)=0$, extending by linearity. We may directly check the generators of $\im(A)$ are contained in the kernels of each of $g_1$, $g_2$, and $g_3$. Thus, there are corresponding induced maps $f_1,f_2,f_3 \in M^{\vee}$. It evident that no $f_i$ is in the $R$-span of the other two, and as $\mu_R(M^{\vee}) \le 3$, it follows $\{f_1,f_2,f_3\}$ is a minimal generating set for $M^{\vee}$.

We now define $\zeta:R^{\oplus 3} \to M^{\vee}$ by $\zeta(e_i)=f_i$ for $i=1,2,3$. We check that the maps $yf_1-xf_2+zf_3$, $zf_1-yf_2+x^2f_3$, $x^2f_1-zf_2+xyf_3$ each map all the generators of $M$ to $0$, and so are the zero map. In particular, $\im(A) \subseteq \ker(\zeta)$, so $\zeta$ induces a map $\phi:M \to M^{\vee}$. As $\{f_1,f_2,f_3\}$ is a minimal generating set for $M^{\vee}$, $\phi$ is obviously surjective. Then as $\phi$ is a surjection between two modules of the same rank, $\ker(\phi)$ is torsion, but $M$ is torsion-free, so $\ker(\phi)=0$ and $\phi$ is injective. 

We observe that the generating sets $\{\bar{e}_1,\bar{e}_2,\bar{e}_3\}$ for $M$ and $\{f_1,f_2,f_3\}$ for $M^{\vee}$ have the property that $f_i(\bar{e}_j)=f_j(\bar{e}_i)$ for all $i,j$ and it follows $M$ is strongly self dual with respect to $I$ through the isomorphism $f$. Proposition \ref{stronglyself} then gives that $\End_R(M)$ has an $R^*$-algebra structure with $M_* \cong M^{\vee}$. 

We finally claim that $M$ is indecomposable.  Indeed, if $M$ is decomposable, then as $M$ is a torsion-free module with $\mu_R(M)=3$ and $\rank(M)=2$, it must decompose into a sum of two ideals, one of which must be principal. Then $M$ has a copy of $R$ as a summand. As $M \cong M^{\vee}$, it would follow that its other summand must be isomorphic to $I$, so we would have $M \cong R \oplus I$. But then Miyata's theorem would force the extension $0 \rightarrow R \xrightarrow{s} M \xrightarrow{t} I \rightarrow 0$ to split, which does not occur. Thus $M$ is indecomposable, and so $E$ is local.

As $M$ fits the construction of the proof of \cite[Proposition 4.7]{HW94}), we have that $M \otimes_R M$ and $M \otimes_R I$ are torsion-free. Then \cite[Lemma 5.3]{DE21} implies that $\Ext^1(I,M^{\vee}) \cong \Ext^1_R(I,M)=0$ and that $\Ext^1_R(M,R)=0$. Applying $\Hom(M,-)$ to the exact sequence 
\[0 \rightarrow R \xrightarrow{s} M \xrightarrow{t} I \rightarrow 0\]
we thus get an exact sequence of right $E$-modules
\[0 \rightarrow \Hom_R(M,R) \xrightarrow{\Hom(M,s)} \Hom_R(M,M) \xrightarrow{\Hom(M,t)} M^{\vee} \rightarrow 0,\]
while applying $\Hom_R(-,M)$ instead gives an exact sequence of left $E$-modules
\[0 \rightarrow \Hom_R(I,M) \xrightarrow{\Hom(t,M)} \Hom_R(M,M) \xrightarrow{\Hom(s,M)} \Hom_R(R,M) \rightarrow 0.\]
In particular, it follows that $\Hom_R(R,M) \cong M$ is a cyclic left $R$-module generated by $s$ while $M^{\vee}$ is a cyclic right $E$-module generated by $t$. In particular, by Nakayama-Azumaya-Krull, $t \otimes s$ is a nonzero element of $M^{\vee} \otimes_E \Hom_R(R,M) \cong M_* \otimes_E M$. But there is a trace map $\phi:M^{\vee} \otimes_E \Hom_R(R,M) \to \Hom_R(R,I)$ given by $f \otimes g \mapsto f \circ g$, and whose kernel is torsion by rank considerations. But then $\phi(t \otimes s)=t \circ s=0$, so $t \otimes s$ is a nonzero torsion element in $M^{\vee} \otimes_E \Hom_R(R,M)$, and in particular $M_* \otimes_E M$ has torsion.

Since $R$ has minimal multiplicity, this example shows as a byproduct that the ideal condition cannot be relaxed to allow modules of higher rank in \cite[Question 4.1]{CG19}, and the condition that $\Hom_R(M,\w) \otimes_R M \cong \w$ in \cite[Question 4.2]{CG19} cannot be relaxed to only assume $\Hom_R(M,\w) \otimes_R M$ is torsion-free or even maximal Cohen-Macaulay.

\end{example}

\bibliographystyle{plain}
\bibliography{mybib}

\vspace{.3cm}

\end{document}